\documentclass[leqno,12pt]{amsart}
\usepackage{amsfonts}
\usepackage{amsmath,amssymb,amsthm}

\newcommand{\R}{\mathbb R}
\newcommand{\E}{\mathbb E}

\newcommand{\tr}{\mathrm{tr}}
\everymath{\displaystyle}

\newtheorem{theorem}{Theorem}

\newtheorem{corollary}[theorem]{Corollary}

\newtheorem{proposition}[theorem]{Proposition}

\begin{document}

\title[Meridian Surfaces on Rotational Hypersurfaces with Lightlike Axis in ${\mathbb E}^4_2$]{Meridian Surfaces on Rotational Hypersurfaces with Lightlike Axis in  ${\mathbb E}^4_2$}

\author{Velichka Milousheva}
\address{Institute of Mathematics and Informatics, Bulgarian Academy of Sciences,
Acad. G. Bonchev Str. bl. 8, 1113, Sofia, Bulgaria;   "L. Karavelov"
Civil Engineering Higher School, 175 Suhodolska Str., 1373 Sofia,
Bulgaria} \email{vmil@math.bas.bg}

\subjclass[2010]{53A35, 53B30, 53B25}
\keywords{Meridian surfaces, pseudo-Euclidean space with neutral metric, constant Gauss curvature, parallel mean curvature vector, parallel normalized mean curvature vector}

\begin{abstract}
We construct a special class of Lorentz surfaces in the
pseudo-Euclidean 4-space with neutral metric which are one-parameter systems of meridians of
 rotational hypersurfaces with lightlike axis and call them meridian surfaces. 
We give the complete classification of the meridian surfaces with constant Gauss curvature and prove  that there are no meridian surfaces with parallel mean curvature vector field other than CMC surfaces lying in a hyperplane. We also classify the meridian surfaces with parallel normalized mean curvature vector field. 
We show that in the family of the meridian surfaces there exist  Lorentz surfaces which have parallel normalized mean curvature vector field but not parallel mean curvature vector.
\end{abstract}

\maketitle

\section{Introduction}

A fundamental problem of the contemporary
differential geometry of surfaces  and hypersurfaces in standard model spaces such as the Euclidean space $\E^n$ and the pseudo-Euclidean space $\E^n_k$ is the investigation of the basic invariants characterizing the surfaces. 
Curvature invariants are the number one Riemannian invariants and the most natural ones. 
The basic curvature invariants of a surface in 4-dimensional Euclidean or pseudo-Euclidean space  are the Gauss curvature and the curvature of the normal connection. The most important normal vector field of a surface is the mean curvature vector field.  So, a fundamental   question is to investigate various important classes of surfaces characterized by conditions on the Gauss curvature, the normal curvature, or the mean curvature vector field, and to find examples of surfaces belonging to these classes. 

Rotational surfaces and hypersurfaces are basic source of examples of many geometric classes of surfaces in Riemannian and pseudo-Riemannian geometry.
The main purpose of this paper is to provide a comprehensive survey on a special class of surfaces (called meridian surfaces) in 4-dimensional Euclidean or pseudo-Euclidean spaces  which are one-parameter systems of meridians of  rotational hypersurfaces. We present briefly recent results on meridian surfaces in the Euclidean space $\E^4$ and the Minkowski space $\E^4_1$.

In  the present paper, the new contribution to the theory of meridian surfaces  is the construction of  2-dimensional Lorentz surfaces in the pseudo-Euclidean space  $\E^4_2$ which are one-parameter systems of meridians of a rotational hypersurface with  lightlike axis.  
They are analogous to the meridian surfaces lying on rotational hypersurfaces with spacelike or timelike axis in $\E^4_2$ which have been studied in  \cite{BM1} and \cite{BM2}. We show that all meridian surfaces are surfaces with flat normal connection and  classify completely  the meridian surfaces   with
constant Gauss curvature (Theorem \ref{T:flat} and Theorem \ref{T:Gauss curvature}). In Theorem \ref{T:parallel-H-p}  we give the classification of the meridian surfaces with parallel mean curvature vector field $H$. Theorem \ref{T:parallel-H0-p} describes all  meridian surfaces  which have parallel normalized  mean curvature vector field but not parallel $H$.

\section{Preliminaries}

Let $\mathbb{E}_{2}^{4}$ be the $4$-dimensional pseudo-Euclidean space with the canonical pseudo-Euclidean metric  of index $2$ given in local coordinates by
\begin{equation*}
\widetilde{g}=dx_{1}^{2}+dx_{2}^{2}-dx_{3}^{2}-dx_{4}^{2}, 
\end{equation*}
where $\left( x_{1},x_{2},x_{3},x_{4}\right) $ is a rectangular coordinate
system of $\mathbb{E}_{2}^{4}$.  Denote by $ \langle ., . \rangle$ the indefinite inner scalar product associated with  $\widetilde{g}$.
Since $\widetilde{g}$ is an indefinite
metric, a vector $v \in \E^4_2$ can have one of
the three casual characters:  \textit{spacelike} if $\langle v, v \rangle >0$
or $v=0$, \textit{timelike} if $\langle v, v \rangle<0$, and \textit{lightlike} if $\langle v, v \rangle =0$ and $v\neq 0$. This terminology is inspired by general relativity and the Minkowski 4-space $\E^4_1$.

We use the following standard denotations:
\begin{equation*}
\begin{array}{l}
\vspace{2mm}
\mathbb{S}^3_2(1) =\left\{V\in \E^4_2: \langle V, V \rangle =1 \right\}; \\
\vspace{2mm}
\mathbb{H}^3_1(-1) =\left\{ V\in \E^4_2: \langle V, V \rangle = -1\right\}.
\end{array}
\end{equation*}
The space $\mathbb{S}^3_2(1)$ is known as the de Sitter space, and the
space $\mathbb{H}^3_1(-1)$ is the hyperbolic space (or the anti-de Sitter space) \cite{O'N}.

 A surface $M$ in $\E^4_2$
 is called \emph{Lorentz}, if $\langle  . , . \rangle$ induces  a Lorentzian
metric $g$ on $M$, i.e. at each point $p\in M$ we have the following decomposition
$$\E^4_2 = T_pM \oplus N_pM$$
with the property that the restriction of the metric
onto the tangent space $T_pM$ is of
signature $(1,1)$, and the restriction of the metric onto the normal space $N_pM$ is of signature $(1,1)$.

Denote by $\nabla $ and $\overline{\nabla}$ the Levi-Civita connections
of $M$ and $\mathbb{E}_{2}^{4}$, respectively. For any tangent vector fields $X,Y$
 and any normal vector
field $\xi$, the Gauss formula and the Weingarten formula are given by
\begin{equation*}
\begin{array}{l}
\vspace{2mm}
\overline{\nabla}_{X}Y=\nabla _{X}Y+h(X,Y),\\
\vspace{2mm}
\overline{\nabla}_{X}\xi =-A_{\xi }X+D_{X}\xi,
\end{array}
\end{equation*}
where $h$ is the second fundamental form of $M$, $D$ is the normal
connection on the normal bundle, and  $A_{\xi}$ is the shape operator with respect to $\xi$.

The mean curvature vector field $H$ of $M$ in $\mathbb{E}_{2}^{4}$
is defined as $H=\frac{1}{2}\, \tr \, h$.
A surface $M$ is called \textit{minimal} if its mean curvature vector vanishes
identically, i.e. $H=0$. A natural extension of minimal surfaces are quasi-minimal surfaces. A surface $M$ is called \textit{quasi-minimal}
(or \textit{pseudo-minimal}) if its mean curvature vector is lightlike at each point,
i.e. $H\neq 0$ and $\left \langle H,H\right \rangle =0$. In the Minkowski space $\E^4_1$ the quasi-minimal surfaces are also called \textit{marginally trapped}. This notion is borrowed from general relativity. A surface 
$M$ is said to have constant mean curvature  if $\langle H, H \rangle = const.$
We shall consider Lorentz surfaces in $\E^4_2$ for which  $\langle H, H \rangle = const \neq 0$. Such surfaces we call \textit{CMC surfaces}.

A normal vector field $\xi$ on $M$ is called \emph{parallel in the normal bundle} (or simply \emph{parallel}) if $D{\xi}=0$ holds identically \cite{Chen2}. A surface $M$ is said to have \emph{parallel mean curvature vector field} if its mean curvature vector $H$
satisfies $D H =0$.

Surfaces for which the mean curvature vector field $H$ is non-zero, $\langle H, H \rangle \neq 0$, and  there exists a unit vector field $b$ in the direction of the mean curvature vector
 $H$, such that $b$ is parallel in the normal
bundle, are called surfaces with \textit{parallel normalized mean curvature vector field} \cite{Chen-MM}.
It is easy to see  that if $M$ is a surface  with non-zero parallel mean curvature vector field $H$ (i.e. $DH = 0$),
then $M$ is a surface with parallel normalized mean curvature vector field, but the converse is not true in general.
It is true only in the case $\Vert H \Vert = const$.

\section{Construction of Meridian Surfaces in Pseudo-Euclidean 4-Space}

Meridian surfaces in the Euclidean 4-space $\E^4$ we defined  in \cite{GM2}  as one-parameter systems of meridians of the
standard rotational hypersurface in $\mathbb{E}^{4}$.
The classification of  meridian surfaces with constant Gauss curvature, with constant mean curvature, Chen meridian surfaces and  meridian surfaces with parallel normal bundle is given in \cite{GM2}  and \cite{GM-BKMS}. The meridian surfaces in $\E^4$ with  pointwise 1-type Gauss map are classified in \cite{ABM}. The idea from the Euclidean space is used in \cite{GM6}, \cite{GM-MC}, and \cite{GM7}  for the construction of meridian spacelike  surfaces lying on rotational hypersurfaces in $\mathbb{E}_{1}^{4}$ with timelike, spacelike, or lightlike axis. The classification of marginally trapped meridian surfaces is given in  \cite{GM6} and \cite{GM7}. 
Meridian surfaces in $\E^4_1$ with pointwise 1-type Gauss map are classified in \cite{AM}. The classification of  meridian surfaces  with constant Gauss curvature, with constant mean curvature, Chen meridian surfaces and  meridian surfaces with parallel normal bundle is given in \cite{GM-MC} and \cite{GM8}.

Following the idea from the Euclidean and Minkowski spaces, in \cite{BM1} and \cite{BM2} we constructed Lorentz meridian surfaces in the pseudo-Euclidean 4-space $\E^4_2$  as one-parameter systems of meridians of rotational hypersurfaces with timelike or
spacelike axis. We gave the classification of quasi-minimal meridian  surfaces and meridian surfaces with constant mean curvature  \cite{BM1}. The classification of meridian surfaces with parallel mean curvature vector field and the classification  of meridian surfaces with parallel normalized mean curvature vector is given in \cite{BM2}.

In the present paper we construct Lorentz meridian surfaces in  $\E^4_2$ which are  one-parameter systems of meridians of rotational hypersurfaces with lightlike axis.

\vskip 2mm
Let $Oe_1 e_2 e_3 e_4$ be a fixed orthonormal coordinate system in $\E^4_2$, i.e. $\langle e_1,e_1 \rangle
=\langle e_2,e_2 \rangle =1, \langle e_3,e_3 \rangle = \langle e_4,e_4 \rangle = -1$. We denote  $\displaystyle{\xi_1= \frac{e_2 + e_4}{\sqrt{2}}},\,\, \displaystyle{\xi_2= \frac{ - e_2 + e_4}{\sqrt{2}}}$ and consider the
 pseudo-orthonormal base $\{e_1, e_3, \xi_1, \xi_2 \}$  of $\mathbb{E}^4_2$.
Note that $\langle\xi_1, \xi_1 \rangle =0$, $\langle \xi_2, \xi_2
\rangle =0$, $\langle \xi_1, \xi_2 \rangle = -1$.

 A rotational
hypersurface with lightlike axis  in $\mathbb{E}^4_2$ can be parametrized by
$$\mathcal{M}: Z(u,w^1,w^2) =  f(u)\, w^1 \cosh w^2 e_1 +  f(u)\, w^1 \sinh w^2 e_3+ \left(f(u) \frac{(w^1)^2}{2} + g(u)\right) \xi_1 + f(u) \xi_2,$$
where $f = f(u), \,\, g = g(u)$ are smooth functions, defined in
an interval $I \subset \mathbb{R}$ and  $f(u)>0, \,\, u \in I$.

Let $w^1 = w^1(v), \, w^2=w^2(v), \,\, v \in J, \,J \subset \R$
and consider the surface $\mathcal{M}_m$ in $\mathbb{E}^4_2$ given by
\begin{equation} \label{E:Eq-1-par}
\mathcal{M}_m: z(u,v) = Z(u,w^1(v),w^2(v)),
\end{equation}
where $u \in I, \, v \in J.$ The surface $\mathcal{M}_m$,
defined by \eqref{E:Eq-1-par}, is a one-parameter system of meridians
of the rotational hypersurface $\mathcal{M}$. So, we call $\mathcal{M}_m$ a \textit{meridian surface
on $\mathcal{M}$}.

Without loss of generality we can assume that $w^1 = \varphi(v), \,
w^2=v$. Then the meridian surface $\mathcal{M}_m$ is parametrized as
follows:
\begin{equation} \label{E:Eq-2-par}
\mathcal{M}_m: z(u,v) = f(u) \left(\varphi(v) \cosh v \,e_1 + \varphi(v) \sinh v \,e_3+ \frac{\varphi^2(v)}{2} \,\xi_1 + \xi_2 \right) + g(u) \, \xi_1.
\end{equation}
If we denote $l(v) = \varphi(v) \cosh v \,e_1 + \varphi(v) \sinh v \,e_3+ \frac{\varphi^2(v)}{2} \,\xi_1 + \xi_2$, then the parametrization \eqref{E:Eq-2-par} is written as 
 \begin{equation*} 
\mathcal{M}_m: z(u,v) = f(u) \,l(v) + g(u) \, \xi_1.
\end{equation*} 

Now we shall find the coefficients of the first  fundamental form of $\mathcal{M}_m$. 
The tangent vector fields $z_u$ and $z_v$ are
\begin{equation} \label{E:Eq-6-par}
\begin{array}{l}
\vspace{2mm}
z_u = \displaystyle{f' \varphi \cosh v \,e_1 + f'
\varphi \sinh v \,e_3+ \left(f' \frac{\varphi^2}{2} +
g'\right)\xi_1 + f' \,\xi_2};\\
\vspace{2mm}
z_v = f( \dot{\varphi} \cosh v + \varphi \sinh v)\,e_1 + f (
\dot{\varphi} \sinh v + \varphi \cosh v) \,e_3+ f \varphi \dot{\varphi} \, \xi_1,
\end{array}
\end{equation}
where $\dot{\varphi}$  denotes the derivative of $\varphi$ with respect to $v$.
So, the coefficients of the first fundamental form are
$$E = - 2 f'(u) g'(u); \quad F = 0; \quad G = f^2(u) (\dot{\varphi}^2(v) - \varphi^2(v)).$$
Since we are studying Lorentz surfaces, in the case $\dot{\varphi}^2(v) - \varphi^2(v) >0$ we assume that $f'(u) g'(u) >0$; in the case  $\dot{\varphi}^2(v) - \varphi^2(v) <0$ we assume that $f'(u) g'(u) <0$.

We consider the tangent frame field defined by 
$X = \displaystyle{\frac{z_u}{\sqrt{2\varepsilon f' g'}}}$,  $Y =  \displaystyle{\frac{z_v}{f \sqrt{\varepsilon (\dot{\varphi}^2 - \varphi^2)}}}$, where $\varepsilon = 1$ in the case $\dot{\varphi}^2 - \varphi^2 >0$, $f' g' >0$, and $\varepsilon = - 1$ in the case $\dot{\varphi}^2 - \varphi^2 <0$, $f' g' <0$.
Thus we have $\langle X, X\rangle = -\varepsilon$, $\langle Y, Y\rangle = \varepsilon$, $\langle X, Y\rangle = 0$.
Let us choose the following  normal frame field:
\begin{equation} \label{E:Eq-7-par}
\begin{array}{l}
\vspace{2mm}
n_1 = \displaystyle{ \sqrt{\frac{\varepsilon f'}{2 g'} }\left( \varphi \cosh v \,e_1 + \varphi \sinh v \,e_3 + \frac{f' \varphi^2 - 2g'}{2f'} \, \xi_1 +
 \xi_2\right)}; \\
\vspace{2mm}
n_2 = \displaystyle{\frac{1}{\sqrt{\varepsilon (\dot{\varphi}^2 - \varphi^2)}}  \left((\dot{\varphi} \sinh v + \varphi \cosh v)\,e_1 + (\dot{\varphi} \cosh v + \varphi \sinh v)\,e_3 +
 \varphi^2  \, \xi_1\right)},
\end{array}
\end{equation}
which satisfies $\langle n_1, n_1 \rangle = \varepsilon$, $\langle n_2, n_2 \rangle = -\varepsilon$, $\langle n_1, n_2 \rangle = 0$.
Taking into account
\eqref{E:Eq-6-par}, we calculate the second partial derivatives of $z(u,v)$:
\begin{equation*}
\begin{array}{l}
\vspace{2mm}
z_{uu} = \displaystyle{f'' \varphi \cosh v \,e_1 + f'' \varphi \sinh v \,e_3+ \left(f''\frac{\varphi^2}{2} +
g''\right)\xi_1 + f'' \,\xi_2};\\
\vspace{2mm}
z_{uv} =  \displaystyle{f'(\dot{\varphi} \cosh v + \varphi \sinh v)\,e_1 + f' (\dot{\varphi} \sinh v  + \varphi \cosh v)\,e_3+ f' \varphi \dot{\varphi} \,\xi_1};\\
\vspace{2mm}
z_{vv} =  f\left( (\ddot{\varphi} + \varphi) \cosh v + 2 \dot{\varphi} \sinh v\right) e_1 + f\left( (\ddot{\varphi}+ \varphi) \sinh v + 2\dot{\varphi} \cosh v\right) e_3+ f\left(\dot{\varphi}^2 + \varphi  \ddot{\varphi}\right)  \xi_1.
\end{array}
\end{equation*}
The last equalities together with  \eqref{E:Eq-7-par} imply
\begin{equation*} 
\begin{array}{ll}
\vspace{2mm}
\langle z_{uu}, n_1 \rangle =  \displaystyle{\frac{f''g' - g'' f'}{\sqrt{2 \varepsilon f' g'}}}; & \qquad \langle z_{uu}, n_2 \rangle = 0;\\
\vspace{2mm}
\langle z_{uv}, n_1 \rangle = 0; & \qquad  \langle z_{uv}, n_2 \rangle = 0;\\
\vspace{2mm}
\langle z_{vv}, n_1 \rangle =  \displaystyle{- f  \sqrt{\frac{\varepsilon f'}{2g'}}\,(\dot{\varphi} ^2 - \varphi^2)}; 
 &
\qquad
\langle z_{vv}, n_2 \rangle =  \displaystyle{f \, \frac{\varphi \ddot{\varphi} - 2 \dot{\varphi}^2 + \varphi^2 }{\sqrt{\varepsilon (\dot{\varphi} ^2 - \varphi^2)}}}.
\end{array}
\end{equation*}
Hence, we obtain
\begin{equation} \label{E:Eq-9-par} 
\begin{array}{l}
\vspace{2mm}
h(X,X) = \varepsilon  \displaystyle{\frac{f''g' - g'' f'}{(2 \varepsilon f' g')^{\frac{3}{2}}}} \,n_1; \\
\vspace{2mm}
h(X,Y) =  0; \\
\vspace{2mm}
h(Y,Y) =   \displaystyle{- \frac{1}{f}  \sqrt{\frac{\varepsilon f'}{2g'}}}\,n_1 - \displaystyle{\varepsilon \frac{\varphi \ddot{\varphi} - 2 \dot{\varphi}^2 + \varphi^2 }{f (\varepsilon (\dot{\varphi} ^2 - \varphi^2))^{\frac{3}{2}}}} \, n_2.
\end{array}
\end{equation}

Now, we shall consider the parametric lines of the
meridian surface $\mathcal{M}_m$. 
The parametric  $u$-line $v = v_0 = const$ is given by
\begin{equation*} 
c_u: z(u) = c \alpha f(u)\,e_1 + c \beta f(u)\,e_3+ \left( \frac{c^2}{2} f(u)
 + g(u)\right)\xi_1 + f(u) \,\xi_2,
\end{equation*}
where $\alpha = \cosh v_0$, $\beta = \sinh v_0$, $c = \varphi(v_0)$. 
So, the unit tangent vector field $t_{c_u}$ of $c_u$ is:
$$t_{c_u}=\frac{1}{\sqrt{2 \varepsilon f' g'}} \left(c \alpha f'\,e_1 + c \beta f'\,e_3+ \left( \frac{c^2}{2} f'
 + g'\right)\xi_1 + f' \,\xi_2\right).$$
We denote by $s$ the arc-length of $c_u$ and calculate the derivative
$$\frac{d t_{c_u}}{d s}=\frac{t_{c_u}'}{s'} = \frac{\varepsilon (f'' g' - g'' f')}{(2 \varepsilon f' g')^2}
\left(c \alpha f'\,e_1 + c \beta f'\,e_3+ \left(\frac{c^2}{2} f'
 - g'\right)\xi_1 + f' \,\xi_2 \right).$$
Thus we obtain that the curvature of $c_u$ is ${ \frac{\varepsilon (f'' g' - g'' f')}{(2 \varepsilon f' g')^{\frac{3}{2}}}}$.
Finally, for each $v = const$ the
parametric lines $c_u$ are congruent in $\E^4_2$. These curves are the meridians of  $\mathcal{M}_m$.
We denote $\kappa_m(u) = { \frac{\varepsilon (f'' g' - g'' f')}{(2 \varepsilon f' g')^{\frac{3}{2}}}}$.

Now, we shall consider the parametric  $v$-lines of $\mathcal{M}_m$.  Let $u = u_0 = const$ and denote $a =f(u_0)$, $b = g(u_0)$.
The corresponding parametric  $v$-line  is given by
\begin{equation*} 
c_v: z(v) = a \varphi(v) \cosh v\,e_1 + a \varphi(v) \sinh v\,e_3+ \left( a\frac{\varphi^2(v)}{2} + b\right)\xi_1 + a \,\xi_2.
\end{equation*}
The unit tangent vector field $t_{c_v}$ of $c_v$ is
$$t_{c_v}=\frac{1}{\sqrt{\varepsilon (\dot{\varphi} ^2 - \varphi^2)}} \left((\dot{\varphi} \cosh v + \varphi \sinh v)\,e_1 + (\dot{\varphi} \sinh v + \varphi \cosh v)\,e_3+
 \varphi \dot{\varphi} \, \xi_1\right).$$
Knowing $t_{c_v}$  we calculate the curvature $\varkappa_{c_v}$ of $c_v$ and obtain that
$\varkappa_{c_v} = {\frac{\varphi \ddot{\varphi} - 2 \dot{\varphi}^2 + \varphi^2 }{ a(\varepsilon (\dot{\varphi} ^2 - \varphi^2))^{\frac{3}{2}}}}.$
We denote $\kappa(v) = {\frac{\varphi \ddot{\varphi} - 2 \dot{\varphi}^2 + \varphi^2 }{(\varepsilon (\dot{\varphi} ^2 - \varphi^2))^{\frac{3}{2}}}}$.
Then, for each  $u = u_0 = const$  the curvature of the corresponding parametric  $v$-line is expressed as
$\varkappa_{c_v} = {\frac{1}{a}\, \kappa(v)}$, where $a =f(u_0)$. Actually, 
$\kappa(v)$ is the curvature of the curve 
$$c: l = l(v) = \varphi(v) \cosh v \,e_1 + \varphi(v) \sinh v \,e_3+ \frac{\varphi^2(v)}{2} \,\xi_1 + \xi_2.$$

Consequently, formulas \eqref{E:Eq-9-par} take the form
\begin{equation} \label{E:Eq-10-par} 
\begin{array}{l}
\vspace{2mm}
h(X,X) = \kappa_m \,n_1; \\
\vspace{2mm}
h(X,Y) =  0; \\
\vspace{2mm}
h(Y,Y) =   \displaystyle{- \frac{1}{f}  \sqrt{\frac{\varepsilon f'}{2g'}}}\,n_1 - \displaystyle{\varepsilon \frac{\kappa}{f}} \, n_2.
\end{array}
\end{equation}
It follows from \eqref{E:Eq-10-par} that the Gauss curvature $K$ of the meridian surface $\mathcal{M}_m$ is expressed as
\begin{equation*} 
K = \varepsilon \frac{\kappa_m}{f}  \sqrt{\frac{\varepsilon f'}{2g'}} 
\end{equation*}
and the mean curvature vector field $H$ is given by
\begin{equation*} 
H = - \frac{\varepsilon}{2}\left( \kappa_m + \frac{1}{f} \sqrt{\frac{\varepsilon f'}{2g'}} \right) n_1 - \frac{\kappa}{2f}\, n_2.
\end{equation*}
Without loss of generality we can assume that $2 \varepsilon f' g' =1$, which implies  $\kappa_m = \frac{f''}{f'}$. Hence,
\begin{equation} \label{E:Eq-Kp}
K= \varepsilon \frac{f''}{f},
\end{equation}
\begin{equation} \label{E:Eq-Hp}
H = - \frac{\varepsilon (f f'' + (f')^2)}{2f f'}\,n_1 - \frac{\kappa}{2f}\, n_2.
\end{equation}

Now, using \eqref{E:Eq-7-par} and \eqref{E:Eq-10-par} we obtain that
\begin{equation}  \label{E:Eq-normal}
\begin{array}{ll}
\vspace{2mm}
\overline{\nabla}_X n_1 = \kappa_m \,X; & \qquad \overline{\nabla}_X n_2 = 0;\\
\vspace{2mm}
\overline{\nabla}_Y n_1 = \frac{1}{f}  \sqrt{\frac{\varepsilon f'}{2g'}} \,Y; & \qquad \overline{\nabla}_Y n_2 = -\varepsilon \frac{\kappa}{f} \,Y.
\end{array}
\end{equation}
Hence, 
\begin{equation}  \label{E:Eq-norm}
\begin{array}{ll}
\vspace{2mm}
D_X n_1 = 0; & \qquad D_X n_2 = 0;\\
\vspace{2mm}
D_Y n_1 = 0; & \qquad D_Y n_2 = 0,\\
\end{array}
\end{equation}
where $D$ is the  normal connection of the surface. The last equalities imply that the curvature of the normal connection of $\mathcal{M}_m$ is zero. 
So, we obtain the following statement.

\begin{proposition}
The meridian surface  $\mathcal{M}_m$, defined by \eqref{E:Eq-2-par}, is a surface with
flat normal connection.
\end{proposition}

\vskip 2mm
In the next sections we will give the classification of the meridian surfaces with constant Gauss curvature, with parallel mean curvature vector field and with parallel normalized mean curvature vector field.

\section{Meridian Surfaces  with constant Gauss curvature}

The study of surfaces with constant Gauss curvature is one of the
essential  topics in differential geometry. Surfaces
with constant Gauss curvature in Minkowski space have drawn the interest of many
geometers, see for example \cite{GalMarMil},  \cite{Lop}, and the references therein.

\vskip 2mm

Let  $\mathcal{M}_m$ be a meridian surface, defined by \eqref{E:Eq-2-par}. The  Gauss curvature of $\mathcal{M}_m$ depends only on the meridian curve $m$ and is expressed by  formula \eqref{E:Eq-Kp}. First, we shall describe the meridian surfaces with zero Gauss curvature.

\begin{theorem} \label{T:flat}
Let $\mathcal{M}_m$ be a meridian surface, defined by \eqref{E:Eq-2-par}.
Then $\mathcal{M}_m$ is flat if and only
if the meridian curve $m$ is given by
$$f(u) = a u + b; \quad g(u) = \frac{\varepsilon}{2a} u + c,$$
where  $a = const \neq 0$, $b=const$, $c = const$. In this case $\mathcal{M}_m$  is a developable ruled surface.
\end{theorem}

\noindent {\it Proof:} 
It follows from \eqref{E:Eq-Kp} that $K =0$ if and only if $f(u) =  a u + b$, $a = const \neq 0$, $b=const$. Using that $2 \varepsilon f' g' =1$, we obtain $g(u) = \frac{\varepsilon}{2a} u + c$, $c = const$. Since in this case $\kappa_m = 0$, then the meridian curve $m$ is part of a straight line, i.e. $\mathcal{M}_m$ lies on  a ruled surface.  
Moreover, it follows from \eqref{E:Eq-normal} that $\overline{\nabla}_X n_1 = 0; \; \overline{\nabla}_X n_2 = 0$, which implies that the normal space is constant at the points of a fixed straight line, and hence the tangent space is one and the same at the points of a fixed line. Consequently, 
$\mathcal{M}_m$  is part of a developable ruled surface.

\qed

The following theorem describes the  meridian surfaces with constant non-zero Gauss curvature.

\begin{theorem} \label{T:Gauss curvature}
Let $\mathcal{M}_m$ be a meridian surface,  defined by \eqref{E:Eq-2-par}.
Then $\mathcal{M}_m$ has constant non-zero Gauss curvature $K$ if and only
if the meridian curve $m$ is given by
\begin{equation} \label{E:Eq-Th}
\begin{array}{ll}
\vspace{2mm}
f(u) = \alpha \cosh \sqrt{\varepsilon K} u + \beta \sinh \sqrt{\varepsilon K} u, & \textrm{if} \quad \varepsilon K >0;\\
\vspace{2mm} f(u) = \alpha \cos \sqrt{- \varepsilon K} u + \beta \sin
\sqrt{- \varepsilon K} u, & \textrm{if} \quad \varepsilon K <0,
\end{array}
\end{equation}
where $\alpha$ and $\beta$ are constants, $g(u)$ is defined by $g'(u) = \frac{\varepsilon}{2f'(u)}$.
\end{theorem}

\noindent {\it Proof:} Using that the Gauss curvature is expressed by  \eqref{E:Eq-Kp}, we obtain that  $K = const \neq 0$  if and
only if the function $f(u)$ satisfies the following differential
equation
$$f''(u) - \varepsilon K f(u) = 0.$$
The general solution of this equation is given by \eqref{E:Eq-Th}, 
where $\alpha$ and $\beta$ are constants.
Since we assume that $2 \varepsilon f' g' =1$, then the function $g(u)$ is
determined by $g'(u) = \frac{\varepsilon}{2f'(u)}$.

\qed

\section{Meridian surfaces  with parallel mean curvature vector field}

Another basic class of surfaces in Riemannian and pseudo-Riemannian geometry are surfaces with parallel mean curvature vector field, since they are critical points of some functionals and  play important role in differential geometry,  the theory of harmonic maps, as well as in physics.
The classification of surfaces with parallel mean curvature vector field in Riemannian space forms was given by Chen \cite{Chen1} and Yau  \cite{Yau}. Recently, spacelike surfaces with parallel mean
curvature vector field  in pseudo-Euclidean spaces with arbitrary codimension  were classified in \cite{Chen1-2} and  \cite{Chen1-3}. The classification of quasi-minimal surfaces with parallel mean
curvature vector in  $\E^4_2$ is given
in \cite{Chen-Garay}.
Lorentz surfaces with parallel mean curvature vector field in arbitrary pseudo-Euclidean space $\E^m_s$ are studied in \cite{Chen-KJM} and \cite{Fu-Hou}. A nice survey on classical and recent results on
submanifolds with parallel mean curvature vector in Riemannian manifolds
as well as in pseudo-Riemannian manifolds is presented in \cite{Chen-survey}.

\vskip 2mm
In this section we shall describe the meridian surfaces with non-zero parallel mean curvature vector field, i.e. $H \neq 0$ and $DH =0$.

\vskip 2mm
Under the assumption $2 \varepsilon f' g' =1$ the mean curvature vector field $H$ of the meridian surface $\mathcal{M}_m$  is given by formula \eqref{E:Eq-Hp}. 
Using that $D_X n_1 = D_Y n_1 = D_X n_2 = D_Y n_2 = 0$, and $X = z_u$,  $Y =  \frac{z_v}{f \sqrt{\varepsilon (\dot{\varphi}^2 - \varphi^2)}}$,  we get
\begin{equation} \label{E:Eq-6p}
\begin{array}{l}
\vspace{2mm}
D_X H =  - \frac{\varepsilon}{2} \left(\frac{f f''+(f')^2}{f f'} \right)' n_1 + \frac{\kappa f'}{2f^2} \,n_2;\\
\vspace{2mm}
D_Y H =  -  \frac{\kappa'}{2f^2 \sqrt{\varepsilon (\dot{\varphi}^2 - \varphi^2)}} \,n_2.
\end{array}
\end{equation}

\begin{theorem} \label{T:parallel-H-p}
Let  $\mathcal{M}_m$ be a meridian surface, defined by \eqref{E:Eq-2-par}. Then
$\mathcal{M}_m$  has parallel mean curvature vector field if and
only if the curvature of $c$ is
$\kappa = 0$ and the meridian curve $m$ is determined by $f' = \phi(f)$ where
\begin{equation*} 
\phi(t) = \frac{a\, t^2 + b}{2 t}, \quad a = const \neq 0, \quad b = const,
\end{equation*}
$g(u)$ is defined by  $g'(u) = \frac{\varepsilon}{2f'(u)}$. In this case  $\mathcal{M}_m$ is a non-flat CMC surface lying in a  hyperplane of $\E^4_2$.
\end{theorem}

\begin{proof}
 Using formulas
\eqref{E:Eq-6p} we get that $\mathcal{M}_m$  has  parallel mean curvature vector field if and only if the following conditions hold
\begin{equation} \label{E:Eq-7p}
\begin{array}{l}
\vspace{2mm}
\left(\frac{f f''+(f')^2 }{f f'} \right)' =0;\\
 \vspace{2mm}
\kappa f' =0;\\
\vspace{2mm}
  \kappa' =0.\\
\end{array}
\end{equation}
Since $f' \neq 0$, the equalities \eqref{E:Eq-7p}  imply that 
$\kappa = 0$ and $\frac{f f''+(f')^2}{f f'}  = a = const$.
If $a = 0$, then $H=0$, i.e. $\mathcal{M}_m$ is minimal. Since we consider non-minimal surfaces, we assume that $a \neq 0$.
In this case the meridian curve $m$ is determined by the following differential
equation:
\begin{equation} \label{E:Eq-8p}
f f''+(f')^2 = a f f', \qquad a = const \neq 0.
\end{equation}
The solutions of the last  differential equation  can be found as follows.
Setting $f' = \phi (f)$ in equation \eqref{E:Eq-8p}, we obtain
that the function $\phi  = \phi (t)$ is a solution of the equation
\begin{equation} \label{E:Eq-9p}
\phi' + \frac{1}{t}\, \phi = a.
\end{equation}
The general solution of equation \eqref{E:Eq-9p} is given by
\begin{equation} \label{E:Eq-90p}
\phi(t) = \frac{a\, t^2 + b}{2 t},  \quad b = const.
\end{equation}

In this case, the mean curvature vector field $H$ is given by $H = -\frac{\varepsilon a}{2} \, n_1$, and thus $\langle H, H \rangle = \frac{\varepsilon a^2}{4} = const$. Hence, the surface $\mathcal{M}_m$  is a CMC surface. Moreover, since $\kappa =0$, from \eqref{E:Eq-normal} it follows that $\overline{\nabla}_X n_2 = 0$, $\overline{\nabla}_Y n_2 = 0$. 
Hence, $\mathcal{M}_m$ lies in a 3-dimensional constant hyperplane  parallel to   $span \{X,Y,n_1\}$. The Gauss curvature  $K \neq 0$, so $\mathcal{M}_m$ is a non-flat CMC surface lying in a hyperplane  of $\E^4_2$. 

\vskip 1mm
Conversely, if the meridian curve $m$ is determined by \eqref{E:Eq-90p}, then by direct computation we get that
$D_X H = D_Y H = 0$, i.e. the surface has parallel mean curvature vector field.

\end{proof}

Theorem \ref{T:parallel-H-p} shows that each meridian surface with parallel mean curvature vector field is  a CMC surface and lies in a hyperplane of $\E^4_2$. So, we have the following result.

\begin{corollary}
There are no Lorentz meridian surfaces with parallel mean curvature vector field other than CMC surfaces lying in a hyperplane of $\E^4_2$.
\end{corollary}

\vskip 2mm
\textit{Remark}. The same result holds true for meridian surfaces lying on rotational hypersurfaces with spacelike or timelike axis \cite{BM2}.

\section{Meridian surfaces  with parallel normalized  mean curvature vector field}

The class of surfaces with parallel mean curvature vector field is naturally extended to the class of surfaces with parallel
normalized mean curvature vector field.
A submanifold in a Riemannian manifold is said to have parallel
normalized mean curvature vector field  if the mean curvature vector is non-zero and the
unit vector in the direction of the mean curvature vector is parallel in the normal
bundle  \cite{Chen-MM}.
It is well known that submanifolds with non-zero parallel mean curvature vector field  have parallel normalized mean curvature vector field.
 But the condition to have parallel normalized mean curvature vector field is much weaker than the condition to have parallel mean curvature vector field.
For example, every surface in the Euclidean 3-space has parallel normalized mean
curvature vector field but in the 4-dimensional Euclidean space, there exist abundant examples of surfaces
which lie fully in $\E^4$ with parallel normalized mean
curvature vector field, but not with parallel mean curvature vector field.
In the pseudo-Euclidean space with neutral metric $\E^4_2$ the study of Lorentz surfaces with  parallel normalized mean
curvature vector field, but not  parallel mean curvature vector field, is still an open problem.

\vskip 2mm
In this section we give the classification of all meridian surfaces  which have parallel normalized  mean curvature vector field but not parallel $H$.

Let  $\mathcal{M}_m$ be a meridian surface, defined by \eqref{E:Eq-2-par}. The mean curvature vector field $H$  is given by formula  \eqref{E:Eq-Hp}. We assume that $\langle H,H \rangle \neq 0$, i.e. $(f f''+(f')^2)^2 - \kappa^2 f'^2 \neq 0$. 

If  $\kappa = 0$, then the  normalized mean curvature vector field is $H_0 = n_1$ and in view of \eqref{E:Eq-norm} we have $D_X H_0 = D_Y H_0 =0$, i.e. $H_0$ is parallel in the normal bundle. We  consider this case as trivial, since under the assumption $\kappa =0$ the surface $\mathcal{M}_m$ lies in a 3-dimensional hyperplane of $\E^4_2$ and every surface in 3-dimensional space has parallel normalized mean curvature vector field. So, further we assume that $\kappa \neq 0$.

A unit   normal vecor field in the direction of $H$ is 
\begin{equation} \label{E:Eq-H0-a}
H_0 =  \frac{- 1}{\sqrt{\left|(f f''+(f')^2)^2 - \kappa^2 f'^2\right|}} \left((f f''+(f')^2) \,n_1  +  \kappa f' \,n_2 \right).
\end{equation}

For simplicity we denote 
$$A = \frac{-(f f''+(f')^2 )}{\sqrt{\left|(f f''+(f')^2 )^2 - \kappa^2 f'^2\right|}}, \quad 
  B = \frac{-\kappa f'}{\sqrt{\left|(f f''+(f')^2 )^2 - \kappa^2 f'^2\right|}},$$
so, the normalized mean  curvature vector field is  expressed as $H_0 = A\, n_1 + B \, n_2$. 
Then from  equalities \eqref{E:Eq-H0-a} and \eqref{E:Eq-norm} we get
\begin{equation} \label{E:Eq-11}
\begin{array}{l}
\vspace{2mm}
D_X H_0 = x(A)\,n_1 + x(B) \,n_2;\\
\vspace{2mm}
D_Y H_0 = y(A)\,n_1 + y(B) \,n_2.
\end{array}
\end{equation}

\begin{theorem} \label{T:parallel-H0-p}
Let  $\mathcal{M}_m$ be a meridian surface,  defined by \eqref{E:Eq-2-par}. Then
$\mathcal{M}_m$  has parallel normalized mean curvature vector field but not parallel mean curvature vector if and
only if one of the following cases holds:

\hskip 10mm (i) $\kappa \neq 0$ and the  meridian  curve $m$ is defined by
$$f(u) = \sqrt{au+b}, \quad g(u) = \frac{2}{3 a^2} (au +b)^{\frac{3}{2}} + c,$$ 
where $a = const \neq 0$, $b=const$, $c = const$.

\hskip 10mm (ii) $\kappa = const \neq 0$ and the meridian curve $m$ is determined by $f' = \phi(f)$ where
\begin{equation} \notag
\phi(t) = \frac{c\, t +b}{t}, \quad c = const \neq 0,\; c^2 \neq \kappa^2, \; b = const,
\end{equation}
$g(u)$ is defined by $g'(u) = \frac{\varepsilon}{2f'(u)}$. 
\end{theorem}

\begin{proof}
Let $\mathcal{M}_m$ be a surface with parallel normalized mean curvature vector field, i.e. $D_X H_0 = 0$, $D_Y H_0 =0$. 
Then from  \eqref{E:Eq-11} it follows that  $A=const$, $B=const$. Hence, 
\begin{equation} \label{E:Eq-12}
\begin{array}{l}
\vspace{2mm}
\frac{-(f f''+(f')^2)}{\sqrt{\left|(f f''+(f')^2)^2 - \kappa^2 f'^2\right|}} = \alpha=const;\\
\frac{-\kappa f'}{\sqrt{\left|(f f''+(f')^2)^2 - \kappa^2 f'^2\right|}} = \beta=const.
\end{array}
\end{equation}

We have the following two cases.

\vskip 1mm
Case (i):  $f f''+(f')^2=0$. In this case, from \eqref{E:Eq-Hp} we get that the mean curvature vector field is $H = - \frac{\kappa}{2f} \,n_2$ and the normalized mean curvature vector field is $H_0 = n_2$. Since we study surfaces with $\langle H,H \rangle \neq 0$, we get $\kappa \neq 0$. The solution of the differential equation $f f''+(f')^2 =0$ is given by the formula 
$f(u) = \sqrt{au+b}$, where $a=const \neq 0$, $b =const$. Using that $g'(u) = \frac{\varepsilon}{2f'(u)}$, we obtain 
$g(u) = \frac{2}{3 a^2} (au +b)^{\frac{3}{2}} + c$, where $c = const$.

\vskip 1mm
Case (ii):  $f f''+(f')^2 \neq 0$ in an interval $\tilde{I} \subset I \subset \mathbb{R}$. Then, from \eqref{E:Eq-12} we get
\begin{equation} \label{E:Eq-13}
\frac{\alpha}{\beta} \,\kappa =  \frac{f f''+(f')^2 }{f'},  \quad \alpha \neq 0, \; \beta \neq 0.
\end{equation}
Since the left-hand side of equality \eqref{E:Eq-13} is a function of $v$, the right-hand side of \eqref{E:Eq-13} is a function of $u$, we obtain that 
\begin{equation} \notag
\begin{array}{l}
\vspace{2mm}
\frac{f f''+(f')^2 }{f'} = c, \quad c = const \neq 0;\\
\kappa = \frac{\beta}{\alpha}\, c.
\end{array}
\end{equation}
In this case  we have $\langle H,H \rangle  = \frac{\varepsilon(c^2 -\kappa^2)}{4f^2}$. 
Since we study surfaces with $\langle H,H \rangle \neq 0$, we get $c^2 \neq \kappa^2$.
The meridian curve $m$ is determined by the following differential
equation:
\begin{equation} \label{E:Eq-14}
f f''+(f')^2 = c f'.
\end{equation}
Setting  $f' = \phi (f)$ in equation \eqref{E:Eq-14}, we obtain
that the function $\phi  = \phi (t)$ satisfies
\begin{equation} \notag
\phi' + \frac{1}{t}\, \phi = \frac{c}{t},
\end{equation}
whose general solution  is  $\phi(t) = \frac{ct + b}{t}$, $b = const$.

\vskip 1mm
Conversely, if one of the cases (i) or (ii) stated in the theorem  holds true, then by direct computation we get that
$D_X H_0 = D_Y H_0 = 0$, i.e. the surface has parallel normalized  mean curvature vector field.
Moreover, in case (i) we have
$$D_XH = \frac{\kappa f'}{2f^2} \,n_2; \qquad D_YH = -  \frac{\kappa'}{2f^2 \sqrt{\varepsilon (\dot{\varphi}^2 - \varphi^2)}} \,n_2,$$
which implies that $H$ is not parallel in the normal bundle, since $\kappa \neq 0$, $f' \neq 0$. 
In case (ii) we get
$$D_XH = \frac{\varepsilon c f'}{2f^2} \,n_1 + \frac{\kappa f'}{2f^2} \,n_2; \qquad D_YH = 0,$$
and again we have that $H$ is not parallel in the normal bundle. 
\end{proof}

\vskip 2mm
\textit{Remark}. Theorem \ref{T:parallel-H0-p}  gives examples of Lorentz surfaces in the pseudo-Euclidean space $\E^4_2$ which have parallel normalized mean curvature vector field   but not parallel mean curvature vector field.

\vskip 5mm \textbf{Acknowledgments:}
The author is partially supported by the Bulgarian National Science Fund,
Ministry of Education and Science of Bulgaria under contract
DFNI-I 02/14.

\end{document}